\documentclass[11pt]{amsart}

\usepackage{amsmath,amssymb,amsthm,mathtools}
\usepackage[margin=1.1in]{geometry}
\usepackage[colorlinks=true,linkcolor=blue,citecolor=blue,urlcolor=blue]{hyperref}

\newtheorem{theorem}{Theorem}
\newtheorem{proposition}[theorem]{Proposition}
\newtheorem{lemma}[theorem]{Lemma}
\newtheorem{corollary}[theorem]{Corollary}
\newtheorem*{liftingsteplemma}{Lemma 6.1 (The lifting step)}
\theoremstyle{remark}

\DeclareMathOperator{\Hilb}{Hilb}
\DeclareMathOperator{\Proj}{Proj}
\DeclareMathOperator{\Spec}{Spec}
\DeclareMathOperator{\Hom}{Hom}
\DeclareMathOperator{\depth}{depth}
\DeclareMathOperator{\coker}{coker}

\title[A one-variable frame construction]{A one-variable frame construction for irrational components of Hilbert schemes of points}
\author{Ruoyu Wu}
\date{}

\begin{document}

\begin{abstract}
Farkas, Pandharipande, and Sammartano constructed non-rational irreducible components of Hilbert schemes of points in affine space $\mathbb A^n$ for all $n\geq 12$. Their construction starts from Hilbert schemes of curves in $\mathbb P^3$, adjoins two auxiliary variables in order to apply Jelisiejew's TNT frame construction, and then doubles the number of variables. We give a one-variable variant of the construction. The new input is a local-cohomology replacement for the depth-three step in Jelisiejew's negative tangent computation. It uses the vanishing of the low-degree Hartshorne--Rao module for the complete $g^3_d$ curve source. As a consequence, over a field of characteristic zero, $\Hilb(\mathbb A^n)$ has non-rational irreducible components for all $n\geq 10$.
\end{abstract}

\maketitle

\section{Introduction}

Let $\Hilb_d(\mathbb A^n)$ be the Hilbert scheme of length $d$ subschemes of affine $n$-space. Farkas, Pandharipande, and Sammartano \cite{FPS} proved that for every $n\geq 12$ there is a $d$ such that $\Hilb_d(\mathbb A^n)$ has an irreducible component which is not rationally connected, and hence not rational.

Their bound comes from the following numerical feature of the construction. One starts with a component of the Hilbert scheme of curves in $\mathbb P^3$, locally described by a graded Hilbert scheme in four variables. To use Jelisiejew's TNT frame construction \cite{Jelisiejew}, FPS adjoin two more variables, raising the number to six; the frame construction then doubles the number of variables, giving $12=2(4+2)$.

We show that, for the particular curve source used in \cite{FPS}, one auxiliary variable is enough. After adjoining one variable, the depth available is only $2$, while Jelisiejew's frame theorem is stated under a depth $3$ hypothesis. The point of the paper is to replace the missing depth-three argument by a local cohomology calculation. The low-degree Hartshorne--Rao module of the FPS curves vanishes, and this rules out the negative bidegrees that would otherwise obstruct the TNT condition.

\begin{theorem}\label{thm:intro}
Let $k$ be an algebraically closed field of characteristic zero. There exists an integer $D$ such that $\Hilb_D(\mathbb A^{10}_k)$ has an irreducible component which is not rationally connected. In particular, it is not rational. Consequently, for every $n\geq 10$ there exists $D_n$ such that $\Hilb_{D_n}(\mathbb A^n_k)$ has a non-rational irreducible component.
\end{theorem}

The proof has three steps. First, we show that adjoining one variable gives a local retraction from the new graded Hilbert scheme back to the original curve source. Second, we prove a local-cohomology vanishing which replaces the depth-three part of Jelisiejew's negative tangent computation. Third, we apply the frame retraction theorem and the FPS finite-truncation comparison to obtain a dominant rational map from the constructed Hilbert-scheme component to $\mathcal M_g$.

\section{The curve source}

We recall the part of the FPS construction that supplies irrationality. Let
\[
C\subset \mathbb P^3
\]
be a general smooth curve of genus $g\geq 22$, embedded by a complete linear series $g^3_d$ in the range used in \cite{FPS}. Let
\[
S=k[x_0,x_1,x_2,x_3],
\qquad
I_C\subset S
\]
be the saturated homogeneous ideal of $C$.

The corresponding component of the Hilbert scheme of curves in $\mathbb P^3$ dominates $\mathcal M_g$. Since $\mathcal M_g$ is of general type for $g\geq 22$ by \cite{HM,EH,FJP}, it is not rationally connected.

We shall use the following elementary consequences of the FPS setup.
\begin{lemma}\label{lem:curve-properties}
For the general curve $C$ above:
\begin{enumerate}
\item $\depth_{S_+}(S/I_C)\geq 1$.
\item $(I_C)_2=0$.
\item $H^1(\mathbb P^3,I_C(t))=0$ for all $t\leq 1$.
\end{enumerate}
\end{lemma}

\begin{proof}
The ideal $I_C$ is saturated, so the first claim follows. The second claim is one of the properties of the FPS curve source: the general curve in the chosen range has no quadrics.

For the third claim, use the exact sequence
\[
0\to I_C(t)\to \mathcal O_{\mathbb P^3}(t)\to \mathcal O_C(t)\to 0.
\]
For $t<0$, both $H^0(\mathbb P^3,\mathcal O(t))$ and $H^0(C,\mathcal O_C(t))$ vanish. For $t=0$, constants restrict isomorphically. For $t=1$, completeness of the $g^3_d$ gives
\[
H^0(\mathbb P^3,\mathcal O(1))
\xrightarrow{\sim}
H^0(C,\mathcal O_C(1)).
\]
Thus $H^1(\mathbb P^3,I_C(t))=0$ for $t\leq 1$.
\end{proof}

\section{Adjoining one variable}

Set
\[
P=S[u],
\qquad
I=I_C P.
\]
Then $\depth_{P_+}(P/I)\geq 2$.

The following retraction replaces the two-variable extension step used in \cite{FPS}.

\begin{lemma}[One-variable graded Hilbert retraction]\label{lem:one-variable-retraction}
Let $h$ be the Hilbert function of $S/I_C$ and define
\[
h'(d)=\sum_{i=0}^d h(i).
\]
Let $H_h(S)$ and $H_{h'}(P)$ be the corresponding graded Hilbert schemes. The map
\[
\theta:H_h(S)\to H_{h'}(P),
\qquad
J\mapsto JP,
\]
admits a local retraction in a neighbourhood of $[I_C P]$.
\end{lemma}

\begin{proof}
Let $B$ be a base scheme and let
\[
K\subset P\otimes_k\mathcal O_B
\]
be a homogeneous ideal corresponding to a $B$-point of $H_{h'}(P)$. Put
\[
Q_d=(P\otimes_k\mathcal O_B/K)_d.
\]
The $Q_d$ are locally free of ranks $h'(d)$. Multiplication by $u$ gives morphisms of vector bundles
\[
u:Q_{d-1}\to Q_d.
\]
Injectivity is an open condition. By bounded regularity on the graded Hilbert scheme, it is enough, in a neighbourhood of $[I_CP]$, to impose it for finitely many $d$. Let $U\subset H_{h'}(P)$ be the open locus where these maps are injective.

For $K\in U$, define
\[
\rho(K)=\frac{K+(u)}{(u)}\subset P/(u)\cong S.
\]
Degree by degree, the quotient by $u$ is
\[
\coker\bigl(u:Q_{d-1}\to Q_d\bigr),
\]
which is locally free of rank
\[
h'(d)-h'(d-1)=h(d).
\]
Thus $\rho(K)$ is a flat family of homogeneous ideals in $S$ with Hilbert function $h$, and this construction defines a morphism
\[
\rho:U\to H_h(S).
\]
For every homogeneous ideal $J\subset S$ one has $\rho(JP)=J$. Hence $\rho\circ\theta=\mathrm{id}$ locally at $[I_C]$.
\end{proof}

\section{The one-variable frame}

Let
\[
T=P[y_0,\ldots,y_4],
\]
with bigrading
\[
\deg x_i=\deg u=(1,0),
\qquad
\deg y_i=(0,1).
\]
Write
\[
\mathfrak m_x=(x_0,x_1,x_2,x_3,u),
\qquad
\mathfrak m_y=(y_0,\ldots,y_4),
\]
and set
\[
Q=x_0y_0+x_1y_1+x_2y_2+x_3y_3+uy_4.
\]
For $a\gg 0$ define
\[
J=IT+\mathfrak m_x^{a+1}+\mathfrak m_y^2+(Q).
\]
This is Jelisiejew's frame with $b=1$, formed after adjoining only one auxiliary $x$-variable. It defines a finite subscheme of $\mathbb A^{10}$.

We prove that $J$ is frame-like in the sense of \cite[Definition 4.2]{Jelisiejew}. The only new point is the replacement of the depth-three part of the negative tangent computation.

\section{The frame criterion used below}

We recall the precise part of Jelisiejew's argument that will be used. This is included to make clear which hypotheses are preserved and which one is replaced.

Let $S'$ be a standard graded polynomial ring, let $I'\subset S'$ be a homogeneous ideal, and let
\[
T'=S'[y_1,\ldots,y_n]
\]
with the usual bigrading. For a frame ideal
\[
J'=I'T'+\mathfrak m_x^{a+1}+\mathfrak m_y^{b+1}+(Q),
\qquad
Q=\sum_i x_i y_i,
\]
Jelisiejew defines the condition of being frame-like by three requirements.
First, $\Spec(T'/J')$ must have trivial negative tangents: the negative part of
\[
\Hom_{T'}(J',T'/J')
\]
is exactly the span of the ambient translations. Second, the degree-zero moving part must be accounted for by the unipotent group sending $y$-variables to $x$-linear forms:
\[
\mathfrak g\longrightarrow
\bigoplus_{\alpha\geq 1}
\Hom_{T'}(J',T'/J')_{(\alpha,-\alpha)}
\]
must be bijective. Third, the ideal must contain a sufficiently high $y$-truncation; in the case $b=1$ this is simply the condition $\mathfrak m_y^2\subset J'$ together with $I'_1=0$.

Once these three conditions hold, \cite[Proposition 4.10]{Jelisiejew} gives a sequence of local retractions from the ordinary Hilbert scheme of points near $[J']$ to the $G_x$-equivariant Hilbert scheme near $[I'+\mathfrak m_x^{a+1}]$. The proof uses two Bialynicki--Birula decompositions and a flag-Hilbert-scheme step. The first retraction uses the TNT condition. The second uses the degree-zero bijectivity above. The last step keeps the $y$-thickening and the quadric $Q$ fixed and remembers only the $x$-graded deformation.

For the ordinary square frame, Jelisiejew proves the TNT condition by splitting negative bidegrees into the following ranges:
\[
(-1,0),\qquad (0,-1),\qquad (-1,-1),\qquad
\alpha\leq -2,\qquad \beta\leq -2.
\]
The degree $(-1,0)$ and $(0,-1)$ computation is \cite[Corollary 3.8]{Jelisiejew}; it requires only
\[
\depth(S'_+,S'/I')\geq2.
\]
The degree $(-1,-1)$ computation is \cite[Lemma 3.9]{Jelisiejew}; it requires $\mathfrak m_x^a\not\subset I'$. The range $\beta\leq-2$ is the $y$-side of \cite[Corollary 3.5]{Jelisiejew}; the necessary depth is supplied by the $y$-variables themselves. The only range in which the depth-three hypothesis on $S'/I'$ is essential for our purposes is the $x$-side range $\alpha\leq-2$.

In the present paper, after adjoining only one variable, we have
\[
\depth(P_+,P/I)\geq2
\]
but not necessarily depth $3$. Thus all pieces listed above remain available except the $x$-side of \cite[Corollary 3.5]{Jelisiejew}. Lemmas~\ref{lem:local-cohomology-vanishing}, \ref{lem:replacement}, and \ref{lem:x-side} are written as a substitute for exactly that missing piece.

Let us also indicate why the replacement has the form of a local-cohomology statement. In the proof of the $x$-side vanishing, a negative tangent in bidegree $(\alpha,\beta)$ with $\alpha\leq-2$ is first restricted to the truncation ideal $\mathfrak m_x^{a+1}T$. The linear syzygies of this truncation reduce the problem to the vanishing of
\[
\Hom_T(\mathfrak m_x^{a+1}T,T/(IT+(Q)))_{(\alpha,\beta)}.
\]
If $T/(IT+(Q))$ had $\mathfrak m_x$-depth at least $2$, this would follow formally from the same Ext-vanishing used by Jelisiejew. With only one auxiliary variable we have depth at least $1$ after quotienting by $Q$, and the possible obstruction is measured by
\[
H^1_{\mathfrak m_x}(T/(IT+(Q))).
\]
Thus the problem becomes one of locating the graded pieces of this local cohomology module.

The Cech computation below shows that the obstruction is built from two factors: the Hartshorne--Rao module of the curve and the module $u^{-1}k[u^{-1}]$ coming from the added variable. The bilinear form
\[
Q=L+uy_4
\]
then forces every nonzero kernel element to have a nonzero $u^{-1}$-coefficient. Consequently, a class in bidegree $(\alpha,\beta)$ would have to come from a Rao class of degree exactly $\alpha$. For the FPS curve source, the relevant Rao groups vanish in degrees $\alpha\leq1$. This is why the replacement proves the stronger-looking vanishing range $\alpha\leq1$, and in particular the range $\alpha\leq-2$ needed for TNT.

\section{The local cohomology replacement}

Put
\[
M=T/(IT+(Q)).
\]

\begin{lemma}\label{lem:local-cohomology-vanishing}
For every $\beta$,
\[
H^0_{\mathfrak m_x}(M_{*,\beta})=0
\]
and
\[
H^1_{\mathfrak m_x}(M)_{\alpha,\beta}=0
\qquad\text{for all }\alpha\leq 1.
\]
\end{lemma}

\begin{proof}
Write
\[
A=S/I_C,\qquad B=A[u],\qquad Y=k[y_0,\ldots,y_4],\qquad U=B\otimes_k Y.
\]
Then $M=U/(Q)$.

Since $I_C$ is saturated, $H^0_{\mathfrak m_S}(A)=0$. Since $u$ is regular on $B=A[u]$, we have $\depth_{\mathfrak m_x}B\geq 2$. The element $Q=L+uy_4$, where
\[
L=x_0y_0+x_1y_1+x_2y_2+x_3y_3,
\]
is a non-zero-divisor on $U$: as a polynomial in $y_4$, its leading coefficient is $u$, a non-zero-divisor on $B$. Hence $\depth_{\mathfrak m_x}M\geq 1$, so $H^0_{\mathfrak m_x}(M)=0$ and therefore $H^0_{\mathfrak m_x}(M_{*,\beta})=0$ for every $\beta$.

From
\[
0\to U(-1,-1)\xrightarrow{\cdot Q}U\to M\to 0
\]
and $H^0_{\mathfrak m_x}(U)=H^1_{\mathfrak m_x}(U)=0$, we obtain
\[
H^1_{\mathfrak m_x}(M)
\cong
\ker\!\left(
Q:H^2_{\mathfrak m_x}(U(-1,-1))\to H^2_{\mathfrak m_x}(U)
\right).
\]

By the Cech-complex Kunneth formula for support in $\mathfrak m_x=\mathfrak m_S+(u)$,
\[
H^2_{\mathfrak m_x}(U)
\cong
H^1_{\mathfrak m_S}(A)\otimes_k H^1_{(u)}(k[u])\otimes_k Y.
\]
The other Kunneth summands in total degree $2$ vanish: indeed
\[
H^0_{\mathfrak m_S}(A)=0,\qquad
H^0_{(u)}(k[u])=0,\qquad
H^i_{(u)}(k[u])=0\quad(i\geq2).
\]
Moreover
\[
H^1_{\mathfrak m_S}(A)
\cong
\bigoplus_t H^1(\mathbb P^3,I_C(t)),
\qquad
H^1_{(u)}(k[u])=u^{-1}k[u^{-1}].
\]

We now keep track of degrees carefully. With the convention
\[
N(-1)_d=N_{d-1},
\]
an element of
\[
H^2_{\mathfrak m_x}(U(-1,-1))_{\alpha,\beta}
\]
is an element of $H^2_{\mathfrak m_x}(U)_{\alpha-1,\beta-1}$. Let $\xi$ be such an element in the kernel of multiplication by $Q$. Under the decomposition above, write
\[
\xi=\sum_{r\geq 1} \xi_r u^{-r}
\]
with finitely many nonzero $\xi_r$, where each $\xi_r$ lies in
\[
H^1_{\mathfrak m_S}(A)\otimes_k Y.
\]
The equation $Q\xi=0$ gives the recurrence
\[
L\xi_r+y_4\xi_{r+1}=0
\qquad(r\geq 1).
\]
If the first nonzero term were $\xi_r$ with $r>1$, then the equation for $r-1$ would give $y_4\xi_r=0$. This is impossible because $Y$ is a polynomial ring and multiplication by $y_4$ is injective on $H^1_{\mathfrak m_S}(A)\otimes_k Y$. Therefore every nonzero kernel element has a nonzero $u^{-1}$-term.

Now suppose that $\xi$ has bidegree $(\alpha,\beta)$ in $H^2_{\mathfrak m_x}(U(-1,-1))$, equivalently bidegree $(\alpha-1,\beta-1)$ in $H^2_{\mathfrak m_x}(U)$. The nonzero $u^{-1}$-coefficient $\xi_1$ has $u$-degree $-1$. Hence its $H^1_{\mathfrak m_S}(A)$-degree is exactly $\alpha$: indeed the total $x$-degree in $H^2_{\mathfrak m_x}(U)$ is $\alpha-1$, and this equals the Rao degree minus $1$. Thus a nonzero kernel element in $H^1_{\mathfrak m_x}(M)_{\alpha,\beta}$ would give a nonzero class in
\[
H^1_{\mathfrak m_S}(A)_\alpha
\cong
H^1(\mathbb P^3,I_C(\alpha)).
\]
By Lemma~\ref{lem:curve-properties}, this group vanishes for $\alpha\leq 1$. Hence
\[
H^1_{\mathfrak m_x}(M)_{\alpha,\beta}=0
\qquad(\alpha\leq 1).
\]
\end{proof}

\begin{lemma}[Replacement for the depth-three $x$-side]\label{lem:replacement}
For all $\beta$ and all $\alpha\leq -2$,
\[
\Hom_T(\mathfrak m_x^{a+1}T,M)_{(\alpha,\beta)}=0.
\]
\end{lemma}

\begin{proof}
Fix $\beta$ and set $N=M_{*,\beta}$ as a graded $P$-module. For fixed $\beta$, this module is finitely generated: it is the $\beta$th $y$-graded piece of the bigraded quotient $M=T/(IT+(Q))$, and it is generated over $P$ by finitely many monomials of $y$-degree $\beta$. Let
\[
\varphi:\mathfrak m_x^{a+1}\to N
\]
be a homogeneous $P$-linear map of degree $\alpha\leq -2$.

Sheafify on $\Proj P$. Since $\mathfrak m_x^{a+1}$ and $P$ have the same sheafification,
\[
\widetilde{\mathfrak m_x^{a+1}}\cong \mathcal O_{\Proj P},
\]
the map $\varphi$ gives a global section of $\widetilde N(\alpha)$. Since $H^0_{\mathfrak m_x}(N)=0$, the standard local-cohomology exact sequence gives
\[
0\to N_\alpha
\to H^0(\Proj P,\widetilde N(\alpha))
\to H^1_{\mathfrak m_x}(N)_\alpha
\to 0.
\]
For $\alpha\leq -2$, the term $N_\alpha$ vanishes because $M$ has no negative $x$-degree terms. The term $H^1_{\mathfrak m_x}(N)_\alpha$ is the bidegree $(\alpha,\beta)$ piece of $H^1_{\mathfrak m_x}(M)$, hence vanishes by Lemma~\ref{lem:local-cohomology-vanishing}. Thus
\[
H^0(\Proj P,\widetilde N(\alpha))=0.
\]

The sheafification of $\varphi$ is therefore zero. Its image is supported at $\mathfrak m_x$, hence is zero because $H^0_{\mathfrak m_x}(N)=0$. Thus $\varphi=0$.
\end{proof}

\begin{liftingsteplemma}
Let
\[
\psi\in \Hom_T(J,T/J)_{(\alpha,\beta)}
\]
with $\alpha\leq -2$. Then the restriction
\[
\psi|_{\mathfrak m_x^{a+1}T}:\mathfrak m_x^{a+1}T\to T/J
\]
is the image of a homomorphism
\[
\widetilde\psi:\mathfrak m_x^{a+1}T\to T/(IT+(Q))=M
\]
of the same bidegree.
\end{liftingsteplemma}

\begin{proof}
This is the lifting argument used in the proof of \cite[Corollary 3.5]{Jelisiejew}; we isolate it because it is the point where possible interactions among the generators of $J$ have to be excluded.

Consider the natural surjection
\[
M=T/(IT+(Q))\longrightarrow T/J.
\]
We have to show that the restriction of $\psi$ to $\mathfrak m_x^{a+1}T$ lifts through this surjection. Since $Q$ has $x$-degree $1$, the image $\psi(Q)$ would have $x$-degree $1+\alpha<0$; hence $\psi(Q)=0$ in $T/J$. For a generator of $\mathfrak m_y^2$, the $x$-degree is $0$, so its image would have $x$-degree $\alpha<0$; hence $\psi(\mathfrak m_y^2)=0$. Thus the only relations to check are the relations among the generators of $\mathfrak m_x^{a+1}T$ after passing modulo $IT+(Q)$.

The ideal $\mathfrak m_x^{a+1}T$ is generated by all monomials of $x$-degree $a+1$, and its first syzygies are generated by the linear relations
\[
x_i m_j-x_j m_i
\]
between such monomial generators. This is the use of Jelisiejew's Lemma 3.1 in the proof of \cite[Corollary 3.5]{Jelisiejew}: after the images of the extra generators $Q$ and $\mathfrak m_y^2$ have vanished, the compatibility with these linear syzygies gives a lift of the restriction to the quotient by $IT+(Q)$. The hypotheses needed for that application are unchanged here: the monomial ideal is the same truncation ideal, and the possible images of $Q$ and $\mathfrak m_y^2$ vanish for the degree reasons above. Therefore the restriction of $\psi$ lifts to the desired map $\widetilde\psi$.
\end{proof}

\begin{lemma}[The $x$-side of Jelisiejew's Corollary 3.5]\label{lem:x-side}
Let
\[
\psi\in \Hom_T(J,T/J)_{(\alpha,\beta)}
\]
with $\alpha\leq -2$. Then $\psi=0$.
\end{lemma}

\begin{proof}
By degree reasons, $\psi(Q)=0$ and $\psi(\mathfrak m_y^2)=0$: the first image would have $x$-degree $1+\alpha<0$, and the second would have negative $x$-degree. By Lemma 6.1, the restriction of $\psi$ to $\mathfrak m_x^{a+1}T$ lifts to a map
\[
\widetilde\psi:\mathfrak m_x^{a+1}T\to M
\]
of the same bidegree. Lemma~\ref{lem:replacement} kills this lift, so $\psi(\mathfrak m_x^{a+1})=0$.

It remains to kill the restriction to $I$. This is exactly the part covered by the I-ignoring lemma \cite[Lemma 3.3]{Jelisiejew}. We use it only in the range $\alpha\leq -1$, and its sole depth hypothesis is
\[
\depth(P_+,P/I)\geq 2,
\]
which holds because $I=I_CP$ and the new variable $u$ is regular on $P/I$. Hence $\psi(I)=0$. The map $\psi$ kills all generators $I$, $\mathfrak m_x^{a+1}$, $\mathfrak m_y^2$, and $Q$ of $J$, so $\psi=0$.
\end{proof}

\section{The frame-like property}

\begin{proposition}\label{prop:frame-like}
For $a\gg0$, the ideal
\[
J=IT+\mathfrak m_x^{a+1}+\mathfrak m_y^2+(Q)
\]
is frame-like.
\end{proposition}

\begin{proof}
We verify the conditions in \cite[Definition 4.2]{Jelisiejew}.

First, $J$ has trivial negative tangents. We spell out the four ranges, since only one of them is changed from the square-frame argument.

For bidegrees $(-1,0)$ and $(0,-1)$, Jelisiejew's Corollary 3.8 applies without alteration. Its proof uses the I-ignoring lemma in depth $2$ and the elementary computations with the equation $Q=\sum x_i y_i$; it does not use the depth-three hypothesis. In our situation $\depth(P_+,P/I)\geq 2$, and there are five pairs of variables, so the corresponding tangent space is spanned by the ten translations
\[
\partial_{x_0},\ldots,\partial_{x_3},\partial_u,
\partial_{y_0},\ldots,\partial_{y_4}.
\]

For bidegree $(-1,-1)$, Jelisiejew's Lemma 3.9 applies. The only extra hypothesis is $\mathfrak m_x^a\not\subset I$. This holds for every $a$, because $I=I_CP$ contains no pure power of the new variable $u$.

For bidegrees with $\alpha\leq -2$, the required vanishing is Lemma~\ref{lem:x-side}. This is the only point where the proof differs from \cite[Corollary 3.5]{Jelisiejew}: the original depth-three vanishing on the $x$-side is replaced by Lemmas~\ref{lem:local-cohomology-vanishing} and \ref{lem:replacement}.

For bidegrees with $\beta\leq -2$, the $y$-side of Jelisiejew's Corollary 3.5 is unchanged. More explicitly, the depth needed there is the depth with respect to the ideal $\mathfrak m_y$. Here depth is understood as the length of a maximal regular sequence contained in $\mathfrak m_y$. The variables $y_0,\ldots,y_4$ form a regular sequence on
\[
T/IT=(P/I)[y_0,\ldots,y_4],
\]
so this ring has $\mathfrak m_y$-depth $5$. Moreover $Q=L+uy_4$ is a non-zero-divisor on $T/IT$: as a polynomial in $y_4$, its leading coefficient is $u$, which is a non-zero-divisor on $P/I$. Since $Q\in\mathfrak m_y$, quotienting by this regular element lowers the $\mathfrak m_y$-depth by at most one. Therefore
\[
\depth_{\mathfrak m_y} T/(IT+(Q))\geq 4.
\]
This verifies the hypotheses used in the $y$-side argument of \cite[Lemma 3.4 and Corollary 3.5]{Jelisiejew}. No depth-three assertion about $P/I$ is used on this side. Hence all negative bidegrees except $(-1,0)$ and $(0,-1)$ vanish, and the latter are exactly the translation directions. Thus $J$ has TNT.

Second, consider the degree-zero tangent condition
\[
\mathfrak g\to
\bigoplus_{\alpha\geq 1}\Hom_T(J,T/J)_{(\alpha,-\alpha)}.
\]
Here $\mathfrak g$ is the Lie algebra of the unipotent group sending the $y$-variables to linear combinations of the $x$-variables. The summand of bidegree $(1,-1)$ is computed in \cite[Corollary 3.15]{Jelisiejew}. That computation is independent of the depth of $P/I$; it uses the square bilinear form, the assumption $b=1$, characteristic not equal to $2$, and the condition $I_2=0$. We are in characteristic zero, $b=1$, and $I_2=0$ by Lemma~\ref{lem:curve-properties}. Therefore the map from $\mathfrak g$ to the $(1,-1)$ summand is bijective.

For $\alpha\geq 2$, the bidegree $(\alpha,-\alpha)$ has $y$-degree at most $-2$, so the $y$-negative vanishing proved above kills these summands. Combining the $(1,-1)$ computation with this vanishing gives exactly the bijectivity required in \cite[Definition 4.2(b)]{Jelisiejew}; this is the content of \cite[Corollary 3.16]{Jelisiejew}, with the $x$-side depth-three input replaced as above.

Finally, the truncation condition in the definition of a frame-like ideal is the standard one for the frame with $b=1$: here $\mathfrak m_y^2=\mathfrak m_y^{b+1}$ and $I_1=0$. Therefore $J$ is frame-like.
\end{proof}

\section{Proof of the main theorem}

Before proving the theorem, we make explicit how the local retractions are used globally. A local retraction of pointed schemes
\[
(X,x)\longrightarrow (Y,y)
\]
means that, after replacing $X$ and $Y$ by open neighbourhoods of the marked points, there are morphisms
\[
\pi:U\to V,\qquad \iota:V\to U
\]
with $\pi\circ\iota=\mathrm{id}_V$. If $V$ is irreducible, then $\pi$ is dominant onto $V$. If $U$ is contained in an irreducible component $Z$ of $X$ and contains the point $\iota(y)$, then after replacing $U$ by $U\cap Z$ the same construction gives a dominant rational map from $Z$ to the irreducible component of $Y$ containing $y$.

We shall apply this observation several times. First, Proposition~\ref{prop:frame-like} and Jelisiejew's retraction theorem pass from the Hilbert scheme of points near $[J]$ to the finite graded Hilbert scheme near the truncation of $I$. Second, the finite-truncation theorem of \cite{FPS} identifies this finite graded Hilbert scheme locally with the ordinary graded Hilbert scheme near $[I]\subset P$. Third, Lemma~\ref{lem:one-variable-retraction} retracts the latter to the graded Hilbert scheme near $[I_C]\subset S$. Finally, the curve-Hilbert component containing $[C]$ maps dominantly to $\mathcal M_g$ by the FPS construction. The composition is a dominant rational map from the component of the Hilbert scheme of points containing $[J]$ to $\mathcal M_g$.

We also use the following standard fact in characteristic zero: if a variety is rationally connected, then every dominant rational image of it is rationally connected. This follows after resolving indeterminacies and using the corresponding statement for dominant morphisms. Since a rational variety is rationally connected, proving that a component is not rationally connected is enough to prove that it is not rational.

\begin{theorem}\label{thm:n10}
There exists an integer $D$ such that $\Hilb_D(\mathbb A^{10}_k)$ has an irreducible component which is not rationally connected. In particular, it is not rational.
\end{theorem}

\begin{proof}
Choose $a$ large enough so that the finite-truncation comparison used in \cite[Theorem 2]{FPS} applies to the graded Hilbert scheme containing $[I]\subset P$, and so that the frame construction above is defined. By Proposition~\ref{prop:frame-like}, $J$ is frame-like. Hence \cite[Proposition 4.10]{Jelisiejew} gives a local retraction from the Hilbert scheme of points near
\[
[J]\in \Hilb(\mathbb A^{10})
\]
to the finite graded Hilbert scheme near the corresponding truncation of $I$.

By the finite truncation theorem of \cite{FPS}, this finite graded Hilbert scheme locally agrees with the graded Hilbert scheme near $[I]\subset P$. By Lemma~\ref{lem:one-variable-retraction}, the latter locally retracts to the graded Hilbert scheme near $[I_C]\subset S$.

Let $\mathcal H_C$ be the irreducible curve-Hilbert component used in \cite{FPS}. The point $[I_C]$ lies on $\mathcal H_C$, and $\mathcal H_C$ admits a dominant rational map to $\mathcal M_g$. Shrinking the open neighbourhoods in the local retractions, we may assume that all targets are irreducible open neighbourhoods of the relevant points on the corresponding components. The composition of the local retractions gives a morphism $\pi$ from an open neighbourhood of $[J]$ in $\Hilb(\mathbb A^{10})$ to an open neighbourhood of $[I_C]$ in $\mathcal H_C$. The composed sections give a morphism $\iota$ in the opposite direction with $\pi\circ\iota=\mathrm{id}$ on this target neighbourhood; hence the image of $\pi$ contains that neighbourhood and $\pi$ is dominant.

Let $\mathcal Z$ be the irreducible component of $\Hilb(\mathbb A^{10})$ containing $[J]$. After replacing the source by the nonempty open subset of $\mathcal Z$ on which the above composition is defined, we obtain a dominant rational map
\[
\mathcal Z\dashrightarrow \mathcal H_C\dashrightarrow \mathcal M_g.
\]
Since $g\geq 22$, the variety $\mathcal M_g$ is of general type by \cite{HM,EH,FJP}; in particular it is not rationally connected. In characteristic zero, the image of a rationally connected variety under a dominant rational map is rationally connected. Therefore $\mathcal Z$ cannot be rationally connected. Hence $\mathcal Z$ is not rational, since rational varieties are rationally connected in characteristic zero.
\end{proof}

\begin{corollary}\label{cor:nge10}
For every $n\geq 10$, there exists $D_n$ such that $\Hilb_{D_n}(\mathbb A^n_k)$ has a non-rational irreducible component.
\end{corollary}

\begin{proof}
For $n>10$, put $m=n-10$ and set
\[
R=T[z_{1},\ldots,z_{m}],
\qquad
J_n=JR+(z_{1},\ldots,z_{m}).
\]
The Bialynicki--Birula fixed-locus argument used in \cite[Corollary 4]{FPS} gives a local retraction from the corresponding component of $\Hilb(\mathbb A^n)$ to the component of $\Hilb(\mathbb A^{10})$ constructed in Theorem~\ref{thm:n10}. The conclusion follows.
\end{proof}

\end{document}